\documentclass[a4paper,11pt]{article}

\usepackage[margin=3cm,footskip=1cm]{geometry}  
\usepackage[T1]{fontenc}   
\usepackage{xr}
    
\usepackage{amsmath,amssymb,amsthm,bm,mathtools,xspace,booktabs,url}
\usepackage{graphicx,etoolbox}
\usepackage[shortlabels]{enumitem}
\usepackage{mathrsfs}
\usepackage{diagbox}
\usepackage{makecell}
\usepackage{appendix}

\usepackage{lmodern} 
\usepackage{dsfont}
\usepackage{microtype}
\usepackage{syntonly}
\usepackage{float}
\usepackage{xcolor}
\usepackage{amsfonts}
\usepackage{color}

\usepackage{etoc}
\usepackage{multirow} 
\usepackage[citecolor=blue,colorlinks=true,linkcolor=blue]{hyperref}
\usepackage{comment}

\usepackage[above,below,section]{placeins} 

\usepackage{algorithm}
\usepackage{algpseudocode}

\setlist{
listparindent=\parindent,
parsep=0pt,
}
      
\usepackage[authoryear,round,sort&compress]{natbib}
    
\AtBeginEnvironment{thebibliography}{%
\small
\setlength\itemsep{0.75em plus 0.5em minus 0.75em}%
}
    
\usepackage{caption}
\usepackage{authblk}

\makeatletter
\newcommand\CorrespondingAuthor[1]{%
\def\@makefnmark{}%
\footnotetext{Corresponding author: #1}%
}
\makeatother

\makeatletter
\renewenvironment{abstract}{%
\small%
\providecommand\keywords{%
\par\medskip\noindent\textit{Keywords:}\xspace}%
\begin{center}%
\bfseries \abstractname\vspace{-.5em}\vspace{\z@}%
\end{center}%
\quote%
}{\endquote}
\makeatother

\newcommand{\de}{\mathrm{d}}
\newcommand{\M}{{\mathcal M}}

\newcommand{\s}{{\mathcal S}}

\newcommand{\x}{{\mathbf x}}
\newcommand{\y}{{\mathbf y}}

\newcommand{\N}{\mathcal N}
\newcommand{\R}{\mathbb R}
\newcommand{\Z}{\mathbb Z}

\newcommand{\X}{{\mathcal X}}
\newcommand{\1}{{\mathbf 1}}
\newcommand{\V}{{\mathcal V}}
\def\P{\mathbb P}
\def\E{\mathbb E}
\DeclareMathOperator{\e}{e}

\newcommand{\beann}{\begin{eqnarray*}}
\newcommand{\eeann}{\end{eqnarray*}}

\newtheorem{thm}{Theorem}
\newtheorem{remark}{Remark}
\newtheorem{conj}{Conjecture}

\graphicspath{{./Figures/}}

\numberwithin{equation}{section}

\begin{document}
    
\title{
New density/likelihood representations for Gibbs models based on  generating functionals of point processes
}

\author{Ottmar Cronie}
\affil{
Department of Mathematical Sciences\\

Chalmers University of Technology \& University of Gothenburg, Sweden\\ 

ottmar@chalmers.se
}

\date{}

\maketitle
    
\begin{abstract}

Deriving exact density functions for Gibbs point processes has been challenging due to their general intractability, stemming from the intractability of their normalising constants/partition functions. This paper offers a solution to this open problem by exploiting a recent alternative representation of point process densities. Here, for a finite point process, the density is expressed as the void probability multiplied by a higher-order Papangelou conditional intensity function. By leveraging recent results on dependent thinnings, exact expressions for generating functionals and void probabilities of locally stable point processes are derived. Consequently, exact expressions for density/likelihood functions, partition functions and posterior densities are also obtained. 
The paper finally extends the results to locally stable Gibbsian random fields on lattices by representing them as point processes.

\keywords 
Exact Bayesian inference,
Exact likelihood inference,
Generating functional, 
Gibbs state, 
Ising model, 
Janossy density, 
Local stability, 
Papangelou conditional intensity, 
Random field, 
Thinning, 
Void probability
\end{abstract}

\section{Introduction}

The study of point processes has long been a topic of significant scientific interest \citep{BRT15,VanLieshoutBook,MW04,DVJ1,DVJ2,baccelli2020}. 
Heuristically, a point process may be described as a generalised random sample where one allows a random sample size and dependence between the sample points. Consequently, they have been instrumental in modelling complex spatial patterns where interactions between points play a crucial role, and they have been exploited in a wide range of applications, in various scientific disciplines. 

Most point processes encountered belong to the family of Gibbs processes. Now, the precise definition of a Gibbs process varies a bit between sources in the literature but a convenient one is assuming the existence of a so-called Papangelou conditional intensity function \citep{MW04,betsch2023structural}, a form of conditional density function, as this is equivalent to the distribution of the point process being a Gibbs measure on the space of point configurations \citep{dereudre2019introduction}. 
Despite their utility, a longstanding challenge 
in the study of Gibbs processes 
has been the derivation of exact density functions. 
Assuming that it exists, the (Janossy) density function of a Gibbs process 
is governed by an energy function determining the probability distribution of the points of the point process  \citep{DVJ1,DVJ2,dereudre2019introduction}. 
The energy function thus governs marginal allotment of points as well as points' interactions, which can be attractive or repulsive, or a combination of both, thereby enabling highly intricate spatial patterns. 
The complexity of these models, coupled with the intricate nature of their interactions, has made it difficult to derive such density functions analytically. 
The reason for this is that the normalising constants of such density functions, so-called partition functions, have not been known in closed form. 
This has been a significant barrier to further advancements in the field, in particular in statistical settings where the intractability of density/likelihood functions has hindered the exploitation of many classical statistical ideas. Existing methods often rely on computationally intensive Monte Carlo simulations or approximations, which can be inaccurate or inefficient, especially in high-dimensional settings \citep{MW04}.

This paper offers an almost complete solution to the long-standing open problem of finding closed forms for densities of point processes, thereby bringing us one step closer to overcoming the limitations that have hindered progress in this area for a very long time. Specifically, by exploiting a recent alternative representation of point process densities by \citet{betsch2023structural}, for a finite point process one can write the density function as the associated void probability, i.e.\ the probability that it has no points, multiplied by a higher-order Papangelou conditional intensity function evaluated in the proposed point configuration and conditioned on the empty point configuration. Since we typically know the Papangelou conditional intensity function of a point process explicitly, what is left is obtaining explicit expressions for void probabilities of point processes. By a result of \citet{last2023disagreement}, any locally stable point process can be obtained as a dependent thinning of a homogeneous Poisson process, where the retention probability is given by the conditional intensity multiplied by an intractable term. Note that a point process is locally stable if its Papangelou conditional intensity is bounded by some finite constant/function \citep{MW04}. 
The next step is to exploit the recent formalisation of thinning of \citet{cronie2024cross}, which tells us that thinnings are 
marginals of binary marked point processes. These observations allow us to exploit standard tools, in particular generating functionals of Poisson processes, to obtain expressions for the void probabilities of locally stable point processes. The obtained void probability expressions are, however, based on the (conditional) expectations of the binary marking field/function used to obtain the thinning. Such expectations are not yet known explicitly, but the author provides a conjecture on its explicit form and thus also conjectures for generating functionals, void probabilities, density/likelihood functions, partition functions and posterior density functions in Bayesian settings. 
Moreover, 
Gibbs processes may be viewed as limits of Gibbsian random fields, where the associated lattices become more and more dense. Such models, which are common in e.g.\ statistical physics, have density functions given by exponentials of energy functions on the lattices in question \citep{georgii2006canonical, van2019theory}. Also here the density functions are intractable due to intractable partition functions. 
When such a random field takes binary values, e.g.\ the Ising model on a finite lattice, we may also represent it as a point process on the lattice in question. Consequently, the results and conjectures obtained for point processes apply also here, provided that the random fields are locally stable, i.e.\ the associated local characteristics (conditional probabilities) are bounded.

\section{Preliminaries}


Consider a general (complete separable metric) space $\s$ with metric $d(\cdot,\cdot)$, most notably (a compact subset of) the $d$-dimensional Euclidean space, $\R^d$, $d\geq1$. 
Throughout, all sets considered will be Borel sets. 
Moreover, endow $\s$ with a suitable Borel reference measure $B\mapsto|B|=\int_B\de u$, e.g.\ Lebesgue measure in the Euclidean setting. As all functions under consideration here will be measurable, they will be introduced and used without emphasising their measurablility. 

Now, let $\X$ be the space of point patterns/configurations $\x=\{x_i\}_{i=1}^n\subseteq\s$, $n\geq0$, which satisfy the local finiteness property that 
$\x(B)=\#(\x\cap B)=\sum_{x\in\x}\1\{x\in B\}<\infty$ for all bounded $B\subseteq\s$. Using that $\x$ is the support of the locally finite counting measure $\x(\cdot)$, we can identify these two entities with each other and equip $\X$ with a Prohorov-type metric, which turns $\X$ itself into a general space \citep{DVJ1,DVJ2}. We denote the induced Borel sets by $\N$. 
Note that, by construction, any element $\x\in\X$ is simple, meaning that $\x(\{u\})\in\{0,1\}$ for any $u\in\s$. 
Given a suitable probability space $(\Omega,\mathcal{F},\P)$, a (simple) point process $X=\{x_i\}_{i=1}^N$ in $\s$ is defined as a random element in $(\X,\N)$, with distribution $P(A)=\P(X\in A)$, $A\in\N$. If the total point count, $N=X(S)$, is almost surely finite, which e.g.\ is the case when $\s$ is bounded, we say that $X$ is a finite point process.

\subsection{Distributional characteristics}

We next provide an overview of a set of distributional characteristics for a point process $X$ in $\s$, which will all play an important role here.

\subsubsection{Generating functionals and void probabilities}

The distribution of $X$ is completely characterised by its (probability) generating functional, which is defined by
\[
G(f) = \E\left[\prod_{x\in X} f(x)\right] = \E\left[\prod_{x\in X} (1-g(x))\right],
\qquad 
f\in \mathcal{V}(\s),
\]
where $\mathcal{V}(\s)$ is the class of functions $f:\s\to\R$ such that $g=1-f:\s\to[0,1]$ has bounded support \citep[Definition 9.4.IV]{DVJ2}. In particular, the void probability distribution $\V(B)$, $B\subseteq\s$, which completely characterises the distribution of $X$ \citep{VanLieshoutBook}, can be expressed by means of $G(\cdot)$ through 
$$
\V(B)=\P(X(B)=0)= \E\left[\prod_{x\in X} \1\{x\notin B\}\right]
= \E\left[\prod_{x\in X} (1-\1\{x\in B\})\right] = G(1-\1\{\cdot\in B\}).
$$


\subsubsection{Papangelou conditional intensities and Gibbs processes}

Another quantity which characterises the distribution of $X$ is the (Papangelou) conditional intensity $\lambda$ \citep{DVJ2}, 
which 
satisfies the Georgii-Nguyen-Zessin (GNZ) formula:
\begin{align}
    \label{eq:GNZ}
\E\left[\sum_{x\in X}
    h(x,X\setminus\{x\})\right]
  =&
     \int_{\s}
     \E
     \left[
     h(u,X)
     \lambda(u;X)
    \right]\de u
\end{align}
for any non-negative (potentially infinite)  $h:\s\times\X_{\s}\to\R$.
Heuristically, $\lambda(u;\x)\de u$ is the conditional probability of having a point of $X$ in an infinitesimal neighbourhood $du$ of $u\in\s$, given that $X$ coincides with $\x$ outside $du$ \citep{VanLieshoutBook}. Following e.g.\ \citet{betsch2023structural}, any point process $X$ with a well-defined conditional intensity is referred to as a Gibbs process.

From the conditional intensity of $X$ we may construct higher-order conditional intensities, which are defined as the permutation-invariant function \citep{betsch2023structural}
\begin{align*}
\lambda^n(x_1,\ldots,x_n;X)
  =
  \lambda(x_1;X)
  \lambda(x_2;X\cup\{x_1\})
  \cdots
  \lambda(x_n;X\cup\{x_1,\ldots,x_{n-1}\})
\end{align*}
for any $n\geq1$ and any $\x=\{x_1,\ldots,x_n\}\subseteq\s$. Here, $\lambda=\lambda^1$ and we will sometimes use the compact form $\lambda^{\#\x}(\x;X)$, $\x\in\X$. If we integrate the $n$th-order conditional intensity with respect to the distribution of $X$, we obtain the associated intensity/product density function 
$$
\rho^n(u_1,\ldots,u_n)=\E[\lambda^n(u_1,\ldots,u_n;X)]
, 
\qquad
\{u_1,\ldots,u_n\}\subseteq\s,
$$
where one often writes $\rho(\cdot)=\rho^1(\cdot)$ and refers to this first-order characteristic as 'the' intensity function. 

The most basic (and important) example of a Gibbs process is a Poisson processes, where $\lambda^n(u_1,\ldots,u_n;\cdot)=\rho^n(u_1,\ldots,u_n)=\rho(u_1)\cdots\rho(u_n)$, $u_1,\ldots,u_n\in\s$, $n\geq1$, and $\V(B)=\exp\{-\int_B\rho(u)\de u\}$. 
Poisson process are, in turn, strongly connected to a particularly tractable class of processes, which are the locally stable ones. These satisfy that \citep{MW04}
$$
\lambda(\cdot;\cdot)\leq\alpha,
$$
for some constant $\alpha<\infty$, whereby a Poisson process with intensity $\alpha$ has a conditional intensity which dominates the conditional intensity $\lambda$.
Moreover, if $\lambda(u;\x)\leq \lambda(u;\y)$ or $\lambda(u;\x)\geq \lambda(u;\y)$, $u\in\s$, whenever $\x\subseteq\y$, we call $\lambda$ and $X$ attractive or repulsive, respectively; note that Poisson processes are both attractive and repulsive. 

\subsubsection{Janossy densities}

Another entity which completely determines that distribution of a point process $X$ is its family of local Janossy densities, $j(\x|B)$, $\x\in\X$, for bounded $B\subseteq\s$. They are the symmetric density functions satisfying \citep{DVJ1, DVJ2}
\[
j(\{x_1,\ldots,x_n\}|B)\de x_1\cdots\de x_n = \P\left(X\cap B=\{x_i\}_{i=1}^N\subseteq\bigcup_{i=1}^n dx_i\right),
\]
given infinitesimal neighbourhoods $\de x_1,\ldots,\de x_n$ of the individual points of $\x=\{x_1,\ldots,x_n\}\in\X$. 
More commonly, in the literature one typically encounters a characterisation of the distribution $P(\cdot|B)=\P(X\cap B\in \cdot)=\int_{\X}\1\{\x\in\cdot\} f(\x|B)\Pi(d\x)$ of $X\cap B$ in terms of its density $f(\cdot|B)$ with respect to the distribution $\Pi$ of a finite Poisson process on $B$ with integrable intensity function $\pi(u)$, $u\in B$. 
These two notions of point process densities are connected by the relationship \citep[Section 10.4]{DVJ2}
\begin{align}
\label{e:PoissonDensity}
f(\x|B)= j(\x|B)\exp\left\{\int_B\pi(u)\de u\right\}.
\end{align} 
For convenience, one here typically lets $\Pi$ be the distribution of a Poisson process with unit intensity, 
whereby $f(\x|B)= j(\x|B)\exp\{|B|\}$. 
Considering a parametrised local Janossy density, 
it is typically expressed on the form \citep[Section 7]{DVJ1}
\[
j_{\theta}(\x|B) = 
\1\{\x\subseteq B\}\frac{\e^{-E_{\theta}(\x|B)}}{\mathcal{Z}_{\theta}^B}, \qquad \x\in\X, \theta\in\theta\subseteq\R^k, k\geq1, 
\]
where $E_{\theta}(\cdot|B)$ is the associated energy function and the normalising constant $\mathcal{Z}_{\theta}^B$, which ensures that $j_{\theta}(\cdot|B)$ is a density function, is commonly referred to as the partition function. 
As one would expect, local densities may be used to obtain the (localised) conditional intensity  \citep{VanLieshoutBook,DVJ2} 
\begin{align}
\label{e:PapangelouFinite}
    \lambda_{\theta}(u;\x|B) 
&= \frac{f_{\theta}((\x\setminus\{u\})\cup\{u\}|B)}
{f_{\theta}(\x\setminus\{u\}|B)}
\\
&= \frac{j_{\theta}((\x\setminus\{u\})\cup\{u\}|B)}
{j_{\theta}(\x\setminus\{u\}|B)}
\nonumber
\\
&=
\exp\left\{-E_{\theta}(\x\cup\{u\}|B) + E_{\theta}(\x|B)\right\}
 \nonumber
\end{align}
of $X\cap B$, assuming that $P(\cdot|B)$ follows the parametrisation $\theta$; this is what allows us to interpret the conditional intensity as a conditional density. 
Note that the partition functions in the numerator and denominator cancel each other.
Unfortunately, in contrast to the conditional intensity, for all but a few trivial examples, the density function is not explicitly known, as a consequence of the partition function not being known explicitly.

\subsection{Maximum likelihood estimation}

Next, consider a finite point process $X$ on $\s$ and write $j_{\theta}(\cdot)=j_{\theta}(\cdot|\s)$, $\theta\in\Theta$. 
Given a point pattern $\x$ in $\s$, which we assume is a realisation from $j_{\theta_0}(\cdot)$ for some unknown $\theta_0\in\Theta$, consider the objective of finding an estimate $\widehat\theta=\widehat\theta(\x)\in\Theta$ of $\theta_0$. Ideally, one would like to solve this by carrying out maximum likelihood estimation, 
i.e.\ 
maximising the 
log-likelihood function 
\citep[Section 7]{DVJ1}
\[
\ell(\theta;\x)
=
\log j_{\theta}(\x)
= \log\frac{\e^{-E_{\theta}(\x)}}{\mathcal{Z}_{\theta}}
=
-E_{\theta}(\x) - \log\mathcal{Z}_{\theta}
, \qquad \theta\in\Theta, \x\in\X.
\]
Unfortunately, the intractability of the partition function, which depends on the model parameters, spills over on the density function, which in turn renders exact maximum likelihood estimation infeasible. For this reason, a range of alternatives have been proposed in the literature \citep{MW04, VanLieshoutBook,BRT15,Coeurjolly2019understanding, cronie2024cross}. 
As the partition functions in the numerator and denominator of the conditional intensity representation cancel each other, many of these alternatives are based on conditional intensities instead of densities. 
Here, the most prominent alternative is arguably pseudolikelihood estimation, where one maximises the log-pseudolikelihood function 
\begin{align}
\label{e:PL}
p\ell(\theta;\x)
=
\sum_{i=1}^{\#\x}
\log\lambda_{\theta}(x_i;\x\setminus\{x\})
-\int_{\s} \lambda_{\theta}(u;\x) \de u
\end{align}
in order to obtain an estimate $\widehat\theta$. When we are considering a Poisson process with $\lambda_{\theta}(u;\cdot)=\rho_{\theta}(u)$, $u\in\s$, then $p\ell(\theta;\x)=\ell(\theta;\x)$, i.e.\ the pseudolikelihood function coincides with the likelihood function. It is however the case that pseudolikelihood estimation does not deliver a satisfactory performance when the interactions in the underlying point process are strong \citep{BRT15}.




    

\section{New representations for generating functionals, density functions, likelihood functions and posterior densities}

We have seen that conditional intensities may be obtained through Janossy densities but one may also ask if the reversed holds true, i.e.\ if we can obtain Janossy densities from conditional intensities. 
Recently, \citet{betsch2023structural} provided an alternative characterisation of local Janossy densities which confirms that this is indeed possible. Considering the parametric setting, let $X_{\theta}$ be a point process on $\s$ which is distributed according to the conditional intensity $\lambda_{\theta}$, $\theta\in\Theta$, 
and consider its void probability
\[
\V_{\theta}(B) = \P(X_{\theta}(B)=0),
\qquad
B\subseteq\s.
\]
By \citet{betsch2023structural} we now have 
\begin{align*}
j_{\theta}(\x|B)
&=
\1\{\x\subseteq B\}
\E[\1\{X_{\theta}(B)=0\}\lambda_{\theta}^{\#\x}(\x;X_{\theta})]
\\
&=
\1\{x_1,\ldots,x_n\in B\}
\E[
\lambda_{\theta}^n(x_1,\ldots,x_n;X_{\theta})|X_{\theta}\cap B=\emptyset]
\V_{\theta}(B)
\nonumber
\end{align*}
for $\x=\{x_1,\ldots,x_n\}\in\X$ and bounded $B\subseteq\s$. 
%
%
In particular, when $\s$ is bounded, so that $X_{\theta}$ is a finite point process, 
we obtain 
\begin{align}
\label{eq:Janossy}
j_{\theta}(\x)
&=
\E[\1\{X_{\theta}(\s)=0\}\lambda_{\theta}^{\#\x}(\x;X_{\theta})]
\\
&=
\lambda_{\theta}^n(x_1,\ldots,x_n;\emptyset)
\V_{\theta}(\s)
\nonumber
\\
&=
  \lambda_{\theta}(x_1;\emptyset)
  \lambda_{\theta}(x_2;\{x_1\})
  \cdots
  \lambda_{\theta}(x_n;\{x_1,\ldots,x_{n-1}\})
  \V_{\theta}(\s)
  \nonumber
\\
&=
\e^{-E_{\theta}(\x) + E_{\theta}(\emptyset)}
\V_{\theta}(\s)
,
\nonumber
\\
\ell(\theta;\x)
&=
\sum_{i=1}^{\#\x}
\log\lambda_{\theta}(x_i;\{x_0,\ldots,x_{i-1}\})
+
\log \V_{\theta}(\s)
\nonumber
\\
&=
-E_{\theta}(\x) + E_{\theta}(\emptyset)
+
\log \V_{\theta}(\s)
\nonumber
\end{align}
where $x_0=\emptyset$, 
whereby the partition function is given by 
\[
\mathcal{Z}_{\theta}
=
\e^{-E_{\theta}(\emptyset)}/\V_{\theta}(\s).
\]
Hence, the density function evaluated in $\x$ and $\theta$ is given by the normalised exponential energy decrease, going from the empty configuration to $\x$. 
Moreover, the representation in \eqref{eq:Janossy} provides us with an important 
algorithmic interpretation of the density: we first account for starting without any points in $\s$ and then we sequentially add points, conditionally on the already added points, according to the conditional density. 
To understand the ramifications of the characterisation \eqref{eq:Janossy}, as already emphasised, for most common models we do not know the explicit form of the density, because of the intractable partition function, but we know the explicit form of the conditional intensity. 
We further see that a key difference between the log-likelihood function and the log-pseudolikelihood function 
is that the latter conditions on all of the data while the former only conditions on parts of it.

More importantly, however, is that 
\eqref{eq:Janossy} 
tells us is that if we manage to explicitly compute the void probability 
for a given (finite) point process, with known conditional intensity, then we also have an exact expression for the density/likelihood function.
To the best of the author's knowledge, the most precise result on the generating functional of a general locally stable Gibbs process has been provided by \citet{stucki2014bounds}, who gave explicit upper and lower bounds for $G(u)$. 
We here improve on their result by providing a semi-explicit characterisation of $G(u)$ for locally stable point processes. 
%
More specifically, 
following \citet{cronie2024cross}, a thinning $X$ of a point process $Y$ is obtained by assigning marks in $\M=\{0,1\}$ to $Y$, according to a (random) marking function/field $M:\s\to\M = \{0,1\}$ yielding $X=\{x\in Y:M(x)=1\}$. We say that $X$ is an (in)dependent thinning of $Y$ in the sense of e.g.\ \citet{iftimi2019second,cronie2016summary,DVJ2} if the marking is done (in)dependently. 
In addition, \citet{last2023disagreement} recently showed that it is indeed possible to obtain a locally stable Gibbs process $X$ as a (dependent) thinning of a Poisson process $Y$, provided that the constant intensity of $Y$ is given by the local stability bound $\alpha$ of $X$, i.e.\ $\lambda(\cdot;\cdot)\leq \alpha$. 
Denote the densities of the finite dimensional distributions of the marking function $M$ by $f_{\M}^n(m_1,\ldots,m_n|x_1,\ldots,x_n)$, $x_1,\ldots,x_n\in\s$, $m_1,\ldots,m_n\in\M$, $n\geq1$ (for any $n\geq1$, the reference measure on $\M^n$ is given by the $n$-fold product of the counting measure on $\M=\{0,1\}$ with itself). Moreover, by Kolmogorov's consistency theorem, $f_{\M}^1(m|u)=\sum_{m_1=0}^1\cdots\sum_{m_n=0}^1 f_{\M}^n(m,m_1,\ldots,m_n|u,x_1,\ldots,x_n)$, $u\in\s$, $m\in\M=\{0,1\}$. We now define the (marginal) retention probability function as 
\begin{align}
\label{e:RetentionProb}
p_{\lambda,\alpha}(u)=f_{\M}^1(1|u)\in[0,1], \qquad u\in\s.
\end{align}
These insights may now be exploited to obtain expressions for generating functionals and void probabilities of locally stable Gibbs processes.



\begin{thm}
\label{thm:GF}
Given a locally stable point process $X$ with conditional intensity $\lambda(\cdot;\cdot)\leq\alpha$, its generating functional satisfies
\begin{align}
\label{e:GF}
G(f) 
=
G(1-g) 
=
\exp\left\{-\int_{\s} g(u) \alpha p_{\lambda,\alpha}(u) \de u \right\}
,
\qquad 
f\in \mathcal{V}(\s).
\end{align}
In particular, for any $B\subseteq\s$, 
\[
\V(B) 
=
\exp\left\{-\int_B \alpha p_{\lambda,\alpha}(u)  \de u \right\}
.
\]

\end{thm}

\begin{proof}
Exploiting the aforementioned thinning insights, by the law of total expectation and e.g.\ \citet[Corollary 2.1.5]{baccelli2020}, 
\begin{align*}
G(f)
&=
\E\left[\prod_{x\in X} (1-g(x))\right]
\\
&=
\E\left[\prod_{x\in Y} (1-\1\{M(x)=1\}g(x))\right]
\\
&=
\E\left[
\prod_{x\in Y} 
(1-\E\left[\1\{M(x)=1\}| Y\right]g(x))
\right]
\\
&=
\E\left[
\prod_{x\in Y} 
(1-\P(M(x)=1| Y)
g(x))
\right]
\\
&=
\E\left[
\prod_{x\in Y} 
(1-f_{\M}(1|x)g(x))
\right]
=
\E\left[
\prod_{x\in Y} 
(1-p_{\lambda,\alpha}(x)g(x))
\right]
=\e^{-\int g(x) \alpha p_{\lambda,\alpha}(x) \de x }
.
\end{align*}


\end{proof} 

Given any $\x=\{x_i\}_{i=1}^n
\in\X$, $\theta\in\Theta$ and $x_0=\emptyset$,
Theorem \ref{thm:GF} now tells us that the density function and the partition function are given by 
\begin{align*}
j_{\theta}(\x)
=&
\lambda_{\theta}(x_1;\emptyset)
  \lambda_{\theta}(x_2;\{x_1\})
  \cdots
  \lambda_{\theta}(x_{\#\x};\{x_1,\ldots,x_{\#\x-1}\})
\exp\left\{-\int_{\s} \alpha p_{\lambda,\alpha}(u;\theta) \de u \right\}
\\
=& 
\exp\left\{
-E_{\theta}(\x) + E_{\theta}(\emptyset)
-\int_{\s} \alpha p_{\lambda,\alpha}(u;\theta) \de u \right\}
\end{align*}
and
\begin{align*}
\mathcal{Z}_{\theta}
=&
\exp\left\{
- E_{\theta}(\emptyset)
+\int_{\s} \alpha p_{\lambda,\alpha}(u;\theta) \de u \right\}
.
\end{align*}

\subsection{A conjecture on the retention probabilities
}
Although Theorem \ref{thm:GF} provides us with an explicit form for the generating functional and the void probability, we do not have an explicit form for the retention probability $p_{\lambda,\alpha}(\cdot)$. Given the results in \citet{last2023disagreement}, what we can say, however, is that it should depend on the conditional intensity. In addition, since $\alpha$ is allowed to be any finite upper bound for the conditional intensity, \eqref{e:GF} should be constant for any $\alpha\geq \sup_{u\in\s,\x\in\X}\lambda(u;\x)$. We further note that we can (independently) thin a Poisson process into another Poisson process, where the intensity of the latter is smaller than the former and the retention probability is given by the ratio of the two intensities \citep{VanLieshoutBook}; the latter is thus locally stable. 
These observations bring us to the following conjecture. 

\begin{conj}
\label{conj}
The function $p_{\lambda,\alpha}(\cdot)$ 
in \eqref{e:RetentionProb} and 
Theorem \ref{thm:GF} 
is given by $p_{\lambda,\alpha}(u)=\lambda(u;\emptyset)/\alpha$, whereby $G(f) 
=
G(1-g) 
=
\exp\left\{-\int_{\s} g(u) \lambda(u;\emptyset) \de u \right\}$ and $\V(B) 
=
\exp\left\{-\int_B \lambda(u;\emptyset) \de u \right\}
$.
\end{conj}

Aside from the purely probabilistic ramifications of this, the statistical impact should not be understated, given that obtaining closed form expressions for likelihood functions for general Gibbs processes, as opposed to resorting to simulation-based approximations \citep{MW04}, has almost been perceived unattainable. 

\begin{remark}
Ongoing work of the author on marked point processes and general thinnings indicates that Conjecture \ref{conj} indeed holds true.
\end{remark}

Under Conjecture \ref{conj}, given $x_0=\emptyset$, the log-likelihood function and the partition function are given by 
\begin{align*}
\ell(\theta;\x)
&=
\log j_{\theta}(\x)
=
\log \lambda_{\theta}^{\#\x}(\x;\emptyset)
+
\log \e^{-\int_{\s} \lambda_{\theta}(u;\emptyset) \de u }
\\
&=
\sum_{i=1}^{\#\x}
\log\lambda_{\theta}(x_i;\{x_0,\ldots,x_{i-1}\})
-\int_{\s} \lambda_{\theta}(u;\emptyset) \de u 
,
\\
\mathcal{Z}_{\theta}
=&
\exp\left\{
- E_{\theta}(\emptyset)
+\int_{\s} \lambda_{\theta}(u;\emptyset) \de u \right\}
.
\end{align*}
In addition, the associated score function here takes the form
\begin{align*}
s(\theta,\x)
&=
\nabla_{\theta} \ell(\theta;\x)
=
\nabla_{\theta}
\sum_{i=1}^{\#\x}
\log\lambda_{\theta}(x_i;\{x_0,\ldots,x_{i-1}\})
-
\int_{\s} 
\alpha p_{\lambda,\alpha}(u;\theta) \de x 
\\
&=
\sum_{i=1}^{\#\x}
\frac{
\nabla_{\theta}
\lambda_{\theta}(x_i;\{x_0,\ldots,x_{i-1}\})
}
{\lambda_{\theta}(x_i;\{x_0,\ldots,x_{i-1}\})}
-\int_{\s} 
\nabla_{\theta}
\alpha p_{\lambda,\alpha}(u;\theta) \de x
\\
&=
\left(-\nabla_{\theta}E_{\theta}(\x) + \nabla_{\theta}E_{\theta}(\emptyset) -\int_{\s} 
\nabla_{\theta}
\alpha p_{\lambda,\alpha}(u;\theta) \de x\right)
j_{\theta}(\x)
, 
\end{align*}
where $\alpha p_{\lambda,\alpha}(u;\theta)=\lambda_{\theta}(x;\emptyset)$ under the conjecture, 
provided that \eqref{e:RetentionProb} is sufficiently nice to be differentiated and to allow for us to interchange the integral $\int_{\s}$ and the gradient operator $\nabla_{\theta}$. 

As for a Poisson process we have $\lambda_{\theta}(x;\cdot)=\rho_{\theta}(x)$, i.e.\ the conditional intensity coincides with the intensity function, we immediately see that the general log-likelihood function form reduces to the well-known Poisson process log-likelihood function. 
To illustrate further, consider the hard-core model \citep{BRT15}, where 
$$
\lambda_{\theta}(u;\x) 
= 
\beta \1\{d(u,x)>R\text{ for all }x\in\x\setminus\{u\}\}
, 
\quad 
u\in\s,\x\in\X,
\theta=(\beta,R)\in(0,\infty)\times[0,\infty); 
$$
note that that $R=0$ yields a Poisson process with intensity $\beta$. Here 
$$
j_{\theta}(\x)
=
\lambda_{\theta}^{\#\x}(\x;\emptyset)
\e^{-\int_{\s} \lambda_{\theta}(u;\emptyset) \de u } = \beta^{\#\x} 
\1\{d(x_i,x_j)>R\text{ for all }x_i,x_j\in\x, x_i\neq x_j\}
\e^{-\beta|\s|}
$$
and the exact energy function can be obtained explicitly by consulting e.g.\ \citet[Section 6.2]{MW04}, which in turn would yield the associated partition function. 
This means that the estimate of the hard-core distance $R$ must be at least as large as the smallest distance between any two points of $\x$. This highlights something addressed by e.g.\ \citet{jansson2024gibbs}, namely that the hard-core model is not identifiable for most statistical approaches. 
Fixing the estimate of the hard-core distance to be in this range, we are left with maximising $\beta\mapsto\log(\beta^{\#\x} \e^{-\beta|\s|}) = \#\x \log \beta - |\s|\beta$, whose derivative with respect to $\beta$ is $\beta\mapsto \#\x/\beta - |\s|$. Setting this to 0 and solving for $\beta$ we obtain the estimate $\#\x/|\s|$ for the abundance parameter $\beta$, which is just the classical estimate of the intensity of a homogeneous point process.

It is also worth writing down the likelihood ratio, given any $\theta_1,\theta_2\in\Theta$, i.e.
\begin{align*}
\frac{j_{\theta_1}(\x)}{j_{\theta_2}(\x)}
=
\frac{\lambda_{\theta_1}^{\#\x}(\x;\emptyset)
\e^{-\int_{\s} \lambda_{\theta_1}(u;\emptyset) \de u }}
{\lambda_{\theta_2}^{\#\x}(\x;\emptyset)
\e^{-\int_{\s} \lambda_{\theta_2}(u;\emptyset) \de u }}
=
\prod_{i=1}^{\#\x}
\frac{\lambda_{\theta_1}(x_i;\{x_0,\ldots,x_{i-1}\})}{\lambda_{\theta_2}(x_i;\{x_0,\ldots,x_{i-1}\})}
\e^{-\int_{\s} \lambda_{\theta_1}(u;\emptyset) - \lambda_{\theta_2}(u;\emptyset)\de u }
.
\end{align*}
This could potentially be used to construct likelihood ratio tests. 

We also note that in the Bayesian setting 
the posterior density of the random parameter $\theta$ satisfies
\begin{align*}
f_{\theta|X}(\beta|\x)
\propto&
j_{\beta}(\x) f_{\theta}(\beta)
\\
=&
    \lambda_{\beta}(x_1;\emptyset)
  \lambda_{\beta}(x_2;\{x_1\})
  \cdots
  \lambda_{\beta}(x_{\#\x};\{x_1,\ldots,x_{\#\x-1}\})
\\
&\times
\exp\left\{-\int_{\s} \alpha p_{\lambda,\alpha}(u;\beta) \de u \right\}
f_{\theta}(\beta)
,
\end{align*}
where $f_{\theta}(\beta)$, $\beta\in\Theta$, is the prior density and $\alpha p_{\lambda,\alpha}(u;\theta)=\lambda_{\theta}(x;\emptyset)$ under the conjecture.

Given the form of the void probability of a Poisson process, 
one may be tempted to think that we can simply replace $\lambda_{\theta}(u;\emptyset)$ by $\rho_{\theta}(u)$ in the conjecture. However, as 
$\rho_{\theta}(u) 
= \E[\lambda_{\theta}(u;X)] = \E[\lambda_{\theta}(u;X)|X_{\theta}=\emptyset]\V_{\theta}(\s) 
+ 
\E[\lambda_{\theta}(u;X_{\theta})|X_{\theta}\neq\emptyset](1-\V_{\theta}(\s))
$, 
we have   
$$
\frac{
\lambda_{\theta}(u;\emptyset)}{\rho_{\theta}(u)} 
=
\V_{\theta}(\s)^{-1} 
-
\frac{\E[\lambda_{\theta}(u;X_{\theta})|X_{\theta}\neq\emptyset]}{\rho_{\theta}(u)}
(\V_{\theta}(\s)^{-1}-1)
.
$$ 
It is clear that the right-hand side is 1 for a Poisson process but this does not hold in general; for an attractive model we have that $\E[\lambda_{\theta}(u;X_{\theta})|X_{\theta}\neq\emptyset]\geq \E[\lambda_{\theta}(u;X_{\theta})] = \rho_{\theta}(u)$, whereby $\lambda_{\theta}(u;\emptyset)\leq\rho_{\theta}(u)$, while for a repulsive one we have that $\E[\lambda_{\theta}(u;X_{\theta})|X_{\theta}\neq\emptyset]\leq \E[\lambda_{\theta}(u;X_{\theta})] = \rho_{\theta}(u)$, whereby $\lambda_{\theta}(u;\emptyset)\geq\rho_{\theta}(u)$. Provided that the conjecture is true, this results in either $\V(B) 
\geq
\e^{-\int_B \rho_{\theta}(u) \de u}$ or $\V(B) 
\leq
\e^{-\int_B \rho_{\theta}(u) \de u}$. 
These observations reveal part of the problem of using the pseudolikelihood function instead of the likelihood function when the underlying model possesses strong interactions.

\section{Locally stable discrete random fields}

Consider a finite collection of 'sites' $\s=\{s_1,\ldots,s_k\}$ equipped with the discrete topology, which makes $\s$ a Polish space and thereby a general space, as well as the counting measure as reference measure $|\cdot|$. Most naturally, we may think of $\s=\Z^d\cap C$ for some bounded $C\subseteq\R^d$, i.e.\ the $d$-dimensional integer lattice $\Z^d$ embedded in the Euclidean $\R^d$ and intersected with $C$; we thus have the metric $d(s_i,s_j)=\|s_i-s_j\|$, $s_i,s_j\in\s$, where $\|\cdot\|$ is the Euclidean norm. 
A point process $X\subseteq\s$, which is automatically finite, thus represents a number of sites which are occupied. If enumerate the sites in some convenient way
and let 
$$
\widetilde{X} = (X_i)_{i=1}^k = (\1\{s\in X\}:s\in\s)\in\M^k=\{0,1\}^k, 
$$
we obtain a discrete random field on $\s$, where each entry takes either value $0$ or $1$. Density functions of such random fields typically have the Gibbs state form \citep{van2019theory}
\[
p_{\theta}^{\s}(\vec{\x}) 
= 
\frac{1}{\widetilde{\mathcal{Z}}_{\theta}}
\exp\left\{-\widetilde{E}_{\theta}(\vec{\x})\right\}
=
\frac{1}{\widetilde{\mathcal{Z}}_{\theta}}
\exp\left\{
\sum_{T\subseteq\s}
V_{\theta}(\vec{\x};T)
\right\},
\qquad
\vec{\x}\in\M^k=\{0,1\}^k,
\theta\in\Theta
,
\]
where the collection of functions $V_{\theta}(\cdot;T):\{0,1\}^k\to\R$, $T\subseteq\s$, with $V_{\theta}(\cdot;\emptyset)=0$ and $V_{\theta}(\vec{\x};T)$ only depending on $\vec{\x}$ restricted to $T$, is called the associated interaction potential. 
The most prominent example is an Ising model, where \citep{van2019theory}
\begin{align*}
\sum_{T\subseteq\s}
V_{\theta}(\vec{\x};T)
=&
\theta_1
\sum_{x\in\vec{\x}}\1\{x=1\}-\1\{x=0\}
\\
&
+
\theta_2
\sum_{x,y\in\vec{\x}:x\sim y}
(\1\{x=1\}-\1\{x=0\})
(\1\{y=1\}-\1\{y=0\})
.
\end{align*}
Here, the parameter $\theta_1$ governs the prevalence of 1:s while $\theta_2$, the so-called inverse temperature, governs the strength of interaction. Moreover, $x\sim y$ means that the sites with values $x$ and $y$ are neighbours,
according to a symmetric neighbourhood relation $\sim$ on $\s$.

As before, the partition function, i.e.\ the normalising constant $\widetilde{\mathcal{Z}}_{\theta}$, is not tractable, which in turn renders the density/likelihood function $p_{\theta}$ intractable. Consequently, a more convenient form of representation is the associated local characteristics \citep{van2019theory}
\[
\tilde{\pi}_i(x_i|\vec{\x}_{-i};\theta)
= 
\frac{p_{\theta}^{\s}(\vec{\x})}
{p_{\theta}^{\s\setminus\{s_i\}}(\vec{\x}_{-i})},
\qquad i\in\s,\vec{\x}\in\M^k=\{0,1\}^k,
\theta\in\Theta,
\]
giving the conditional density for the value at any location $i$, given the values at all other sites. 
Here, $p_{\theta}^{\s\setminus\{s_i\}}(\cdot)$ is a density on the sites $\s\setminus\{i\}$ (under suitable regularity) and $\vec{\x}_{-i}$ is $\vec{\x}$ with the entry on position $i$ removed. 
In particular, for an Ising model we have that \citep{van2019theory}
\[
\tilde{\pi}_i(x_i|\vec{\x}_{-i};\theta)
= \frac{\exp\left\{\theta_1 + 
\theta_2\sum_{y\in\vec{\x}_{-i}:x_i\sim y}
(\1\{y=1\}-\1\{y=0\})
\right\}}
{1 + \exp\left\{\theta_1 + 
\theta_2\sum_{y\in\vec{\x}_{-i}:x_i\sim y}
(\1\{y=1\}-\1\{y=0\})
\right\}}.
\]
Writing $(s_i,x_i)$, $i=1,\ldots,k$, for the site-value pairs of $\s$ and $\vec{\x}$, by \eqref{e:PapangelouFinite} we have that $\tilde{\pi}_i(1|\vec{\x}_{-i};\theta) = \lambda_{\theta}(s_i;\{s_j\in\s:x_j\in\vec{\x}_{-i},x_j=1\})$ and $\tilde{\pi}_i(0|\vec{\x}_{-i};\theta) = 1 - \lambda_{\theta}(s_i;\{s_j\in\s:x_j\in\vec{\x}_{-i},x_j=1\})$, assuming that $\lambda_{\theta}$ is the conditional intensity of $X$. 
We further see that locally stability of $X$ implies that all local characteristics are bounded by $\alpha$. Hence, under Conjecture \ref{conj}, the associated log-likelihood function is given by 
\begin{align*}
\ell(\theta;\vec{\x})
&=
\sum_{i=1}^{\#\x}
\log\lambda_{\theta}(x_i;\{x_0,\ldots,x_{i-1}\})
-\int_{\s} \lambda_{\theta}(u;\emptyset) \de u 
\\
&=
\sum_{i=1}^{k}
\1\{x_i=1\} \log \tilde{\pi}_i(1|
(x_j:1\leq j\leq i, x_j=1)
;\theta)
- 
\tilde{\pi}_i(1|(0,\ldots,0);\theta)
,
\end{align*}
where $\x$ is the set consisting of the entries of $\vec{\x}$. 
Note that the integral becomes a sum here and that $(x_j:1\leq j\leq i, x_j=1)$ is a vector of dimension less than $k$ while in the second term $(0,\ldots,0)$ is a vector of zeros of length $k-1$. Consequently, we consider conditional density models on reduced site sets to build the likelihood function. Moreover, by construction, this function is invariant to the order in which we label the sites.

\begin{remark}
If the values of $\widetilde X$ are not necessarily binary, e.g.\ $\M=\{0,1,\ldots,l\}$ as in a Potts model or $\M$ given by some interval, we can follow the same procedure and set the smallest value of $\M$ as a reference value, which we identify with no occupation, while all other states are identified with occupation. To include the specific values at all sites, we assign these as marks to the points of $X$, using $\M$ as mark set. Note that when $\M$ is an interval we almost surely have occupation at all sites. Moreover, this would require the use of Papangelou conditional intensities for marked point processes.

\end{remark}



\section{Discussion}

Gibbs/Boltzmann densities, which date back to the mid 1800's, have had a significant use and interest in a range of different fields, including spatial statistics and statistical physics. A major hurdle in their study, however, has been the intractability of their normalising constants, the so-called partition functions. 
This work provides a general solution to obtaining closed-form expressions for densities of finite Gibbs point processes. To achieve this, the paper exploits recent a recent representation of density functions for point processes along with recent results on representing locally stable point processes as thinnings of Poisson processes. 
The paper further shows that the proposed solution may also be applied to certain Gibbs fields, most notably Ising models on finite lattices. 

Overall, this work fills an important gap and takes a significant step toward making Gibbs distributions more accessible and usable, by removing the limitation of having densities only specified up to normalising constants. The ability to work directly with normalised probability densities should enable e.g.\ new modelling approaches and analyses.
It is, however, hard to anticipate what the exact ramifications of the new findings will be, but the author is hopeful that these new advancements will contribute to opening up new avenues in the study of spatial random models, both in applied and theoretical settings. 
Personally, the author has a particular interest in evaluating how likelihood estimation performs in comparison to the recent Point Process Learning methodology \citep{cronie2024cross,jansson2024comparison}. 

\section{Acknowledgements}
The author is grateful to Christian Hirsch, Mathis Rost, Aila Särkkä and Moritz Schauer for fruitful discussions.
This research has been supported by the Swedish Research Council (2023-03320).

\pagebreak

\bibliographystyle{dcu}
\bibliography{Refs.bib}

\end{document}